\documentclass[11pt,reqno]{amsart}

\usepackage{amsmath,amssymb,amsfonts,dsfont}
\usepackage{cite,graphicx,xcolor,hyperref}

\newtheorem{Proposition}{Proposition}
\newtheorem{Theorem}{Theorem}

\newcommand{\me}{\mathrm{e}}
\newcommand{\mi}{\mathrm{i}}

\author{K.\,A. Rybakov}

\title[On the orthogonal expansion of iterated Stratonovich stochastic integrals]{On the orthogonal expansion of\\iterated Stratonovich stochastic integrals}

\begin{document}

\maketitle

\begin{center}

\vskip -3.5ex

Moscow Aviation Institute (National Research University);\\
125993, Moscow, Volokolamskoe Hwy, 4;\\
rkoffice@mail.ru
\end{center}

\vskip 2.5ex

\textbf{Abstract.} We consider a class of functions for which the multiple Stratonovich stochastic integral or equivalent iterated Stratonovich stochastic integral with square integrable weights is defined by the orthogonal expansion. The equality of the trace of expansion coefficients matrix for these functions and the corresponding integral trace is established.

\vskip 0.5ex

\textbf{Keywords:} trace class operators, multiple Stratonovich stochastic integrals, iterated Stratonovich stochastic integrals, orthogonal expansion

\vskip 0.5ex

\textbf{MSC:} 47B10; 60H35

\makeatletter{\renewcommand*{\@makefnmark}{}
\footnotetext{This is a translation of the article: Rybakov, K.A. On the orthogonal expansion of the iterated Stratonovich stochastic integrals. {\em Herald\;of\;Dagestan\;State\;University. Series 1. Natural Sciences} {\bf 2022}, {\em 37(2)}, 27--32. (in Russian) \linebreak \url{https://doi.org/10.21779/2542-0321-2022-37-2-27-32}}}

\renewcommand{\thefootnote}{\fnsymbol{footnote}}

\thispagestyle{empty}

\section{Introduction}\label{secIntro}

The article presents a theorem that is of great importance for the representation of iterated Stratonovich stochastic integrals over Wiener processes with applications to solving numerically systems of stochastic differential equations~\cite{Ref1, Ref2}.

In the orthogonal expansion of multiple and iterated Stratonovich stochastic integrals of multiplicity $k$ for functions from $L_2(\mathds{T}^k)$, $\mathds{T} = [t_0,T]$, it becomes necessary to consider a special class of functions, namely the linear subspace of $L_2(\mathds{T}^k)$, the space of real-valued square-integrable functions. For functions from this class, integral traces are defined over different pairs of variables, as well as for expansion coefficients relative to basis functions (all possible products of basis functions of $L_2(\mathds{T})$), matrix traces are defined over the corresponding pairs of indices. In~\cite{Ref4}, the special space $L_2^{\text{tr}(j_1 \dots j_k)}(\mathds{T}^k)$ is introduced for this purpose, where $j_1,\dots,j_k$ are some natural numbers, the indices of independent Wiener processes~\cite{Ref1, Ref2, Ref3, Ref4}.

Consider an example for $k = 2$. Let $f \in L_2(\mathds{T}^2)$, $\{q_i\}$ is a basis of $L_2(\mathds{T})$, and $F_{ij}$ are expansion coefficients of the function $f$ relative to the basis $\{q_i q_j\}$:
\begin{equation}\label{eq1}
  F_{ij} = \int_{\mathds{T}^2} {f(t,\tau) q_i(t) q_j(\tau) dt d\tau}, \ \ \ i,j = 0,1,2,\ldots
\end{equation}
Then, for the correct definition of the multiple Stratonovich stochastic integral for the function~$f$ (one of two possible variants) and its orthogonal expansion, the following condition should be satisfied:
\[
  \sum\limits_{i=0}^\infty {F_{ii}} = \int_\mathds{T} {f(t,t) dt},
\]
where the left-hand side is the matrix trace and the right-hand side is the integral trace. In this case, $f \in L_2^{\text{tr}(j_1 j_2)}(\mathds{T}^2)$ for $j_1 = j_2$. It is well known that such a condition is not satisfied for arbitrary functions $f \in L_2(\mathds{T}^2)$. Additionally, the integral on the right-hand side must be considered for continuous functions or interpreted in a specific manner by applying an averaging operator for the function $f$, which can be defined differently on the diagonal of the square $\mathds{T}^2$ in $L_2(\mathds{T}^2)$, to specify a function that is uniquely defined on $\mathds{T}^2$ or on the diagonal and coincides with the original one almost everywhere~\cite{Ref5, Ref6}.

For arbitrary $k > 2$, it is necessary to consider similar conditions for all possible pairs or for some pairs of indices and variables~\cite{Ref7, Ref8}: such pairs arise under the condition that the integration is carried out over coinciding Wiener processes (for multiple or iterated Stratonovich stochastic integrals).

Further, we focus on the function $f(t,\tau) = \varphi(t) \psi(\tau) 1(t-\tau)$, where $1(t-\tau)$ is a unit step function. The multiple stochastic integral for such a function coincides with the iterated stochastic integral, which corresponds to a second-order system of linear stochastic differential equations. In this context, $\varphi$ and $\psi$ are weight functions. In~\cite{Ref1, Ref2}, the equality
\[
  \sum\limits_{i=0}^\infty {F_{ii}} = \frac{1}{2} \int_\mathds{T} {\varphi(t) \psi(t) dt}
\]
is proved under the condition that $\varphi$ and $\psi$ are continuously differentiable functions, and $\{q_i\}$ are Legendre polynomials or trigonometric functions (Fourier basis).\footnote{See Appendix.} In this article, such an equality is proved for $\varphi,\psi \in L_2(\mathds{T})$ and an arbitrary basis $\{q_i\}$.

\section{Necessary definitions and preliminary results}\label{secPreliminary}

We will consider linear operators $\operatorname{F}$ in the space $L_2(\mathds{T})$, $\mathds{T} = [t_0,T]$, determined by the relation
\[
  \operatorname{F} \! g(t) = \int_\mathds{T} {f(t,\tau) g(\tau) d\tau} \ \ \ \forall g \in L_2(\mathds{T}),
\]
where $f \in L_2(\mathds{T}^2)$ is the operator kernel. This operator is called a trace class operator~\cite{Ref6, Ref9} if there exist functions $f_1,f_2 \in L_2(\mathds{T}^2)$ that
\begin{equation}\label{eq2}
  f(t,\tau) = \int_\mathds{T} {f_1(t,\xi) f_2(\xi,\tau) d\xi}.
\end{equation}

\begin{Theorem}[see~\cite{Ref5, Ref6}]\label{thm1}
Let the linear operator $\operatorname{F} \colon L_2(\mathds{T}) \to L_2(\mathds{T})$ with the kernel $f \in L_2(\mathds{T}^2)$ is a trace class operator, and $\{q_i\}$ is a basis of $L_2(\mathds{T})$. Then
\begin{equation}\label{eq3}
  \sum\limits_{i=0}^\infty F_{ii} = \int_\mathds{T} {f(t,t) dt}, 
\end{equation}
where $F_{ij}$ are expansion coefficients~\eqref{eq1} of the function $f$ relative to the basis $\{q_i q_j\}$, and the integral on the right-hand side is understood as follows:
\[
  \int_\mathds{T} {f(t,t) dt} = \lim\limits_{\varepsilon \to 0} \int_\mathds{T} {\operatorname{S}_\varepsilon f(t,t) dt}, \ \ \
  \operatorname{S}_\varepsilon f(t,\tau) = \frac{1}{4\varepsilon^2} \int_{\begin{smallmatrix} \hfill |t-\theta| \leqslant \varepsilon \\ |\tau-\vartheta| \leqslant \varepsilon \end{smallmatrix}} {f(\theta,\vartheta) d\theta d\vartheta},
\]
i.e., $\operatorname{S}_\varepsilon$ is the averaging operator. For the last relation, the function $f$ is extended by zero outside the square $\mathds{T}^2$.
\end{Theorem}

Note that the series on the right-hand side of the formula~\eqref{eq3} converges absolutely, and its sum does not depend on the choice of a basis $\{q_i\}$.

\begin{Proposition}\label{prp1}
Linear operators $\operatorname{F}$ with kernels
\begin{align}
  f(t,\tau) = t^n \tau^{m+n} 1(t-\tau) + \tau^n t^{m+n} 1(\tau-t) = f(\tau,t), \label{eq4} \\
  f(t,\tau) = t^{m+n} \tau^n 1(t-\tau) + \tau^{m+n} t^n 1(\tau-t) = f(\tau,t), \label{eq5}
\end{align}
where $m \in \mathds{N}$ and $n \in \{0,1,2,\dots\}$, are trace class operators.
\end{Proposition}

\begin{proof}
Let $f_1(t,\tau) = m t^n \tau^{m-1} 1(t-\tau)$ and $f_2(t,\tau) = \tau^n 1(\tau - t)$. By using the representation~\eqref{eq2}, we have
\begin{align*}
  f(t,\tau) & = m t^n \tau^n \int_\mathds{T} {\xi^{m-1} 1(t-\xi) 1(\tau-\xi) d\xi} = m t^n \tau^n \int_{t_0}^{\min\{t,\tau\}} {\xi^{m-1} d\xi} \\
  & = t^n \tau^n \min\{t^m,\tau^m\} - t^n \tau^n t_0^m = t^n \tau^{m+n} 1(t-\tau) + t^{n+m} \tau^n 1(\tau-t) - t^n \tau^n t_0^m.
\end{align*}

The last term $\tau^n t^n t_0^m$ defines a trace class operator with the simplest kernel; hence, the function~\eqref{eq4} defines a trace class operator (such operators form a linear space~\cite{Ref9}).

Next, let $f_1(t,\tau) = m t^n \tau^{m-1} 1(\tau - t)$ and $f_2(t,\tau) = \tau^n 1(t-\tau)$. Then
\begin{align*}
  f(t,\tau) & = m t^n \tau^n \int_\mathds{T} {\xi^{m-1} 1(\xi-t) 1(\xi-\tau) d\xi} = m t^n \tau^n \int_{\max\{t,\tau\}}^T {\xi^{m-1} d\xi} \\
  & = t^n \tau^n T^m - t^n \tau^n \max\{t^m,\tau^m\} = t^n \tau^n T^m - t^{n+m} \tau^n 1(t-\tau) - \tau^{m+n} t^n 1(\tau-t).
\end{align*}

By similar reasoning, one can prove that the function~\eqref{eq5} also defines a trace class operator.
\end{proof}

\section{The main result}\label{secMain}

In this section, we will present the main result of this article in the form of a theorem.

\begin{Theorem}\label{thm2}
Let $\varphi,\psi \in L_2(\mathds{T})$, and $\{q_i\}$ is a basis of $L_2(\mathds{T})$. Then
\begin{equation}\label{eq6}
  \sum\limits_{i=0}^\infty \int_\mathds{T} {\varphi(t) q_i(t) \int_{t_0}^t {\psi(\tau) q_i(\tau) d\tau} dt} = \frac{1}{2} (\varphi,\psi)_{L_2(\mathds{T})}.
\end{equation}
\end{Theorem}

\begin{proof}
Consider functions $g,g^* \in L_2(\mathds{T}^2)$:
\[
  g(t,\tau) = \varphi(t) \psi(\tau) 1(t-\tau), \ \ \ g^*(t,\tau) = \psi(t) \varphi(\tau) 1(\tau-t) = g(\tau,t).
\]

For these functions, expansion coefficients $G_{ij},G_{ij}^*$  relative to the basis $\{q_i q_j\}$ are determined by the formulae:
\[
  G_{ij} = \int_\mathds{T} {\varphi(t) q_i(t) \int_{t_0}^t {\psi(\tau) q_j(\tau) d\tau} dt}, \ \ \ G_{ij}^* = \int_\mathds{T} {\psi(t) q_i(t) \int_t^T {\varphi(\tau) q_j(\tau) d\tau} dt},
\]
and for them the condition $G_{ij} = G_{ji}^*$, $i,j = 0,1,2,\dots$, is satisfied.

Further, let
\[
  f(t,\tau) = g(t,\tau) + g^*(t,\tau) = g(t,\tau) + g(\tau,t) = f(\tau,t), \ \ \ f \in L_2(\mathds{T}^2),
\]
then expansion coefficients $F_{ij}$ of the function $f$ are determined by the linearity:
\[
  F_{ij} = G_{ij} + G_{ij}^* = G_{ij} + G_{ji}, \ \ \ F_{ii} = 2 G_{ii}.
\]

Additionally, we can write the integral trace for the function $f$:
\[
 \lim\limits_{\varepsilon \to 0} \operatorname{S}_\varepsilon f(t,t) = \varphi(t) \psi(t), \ \ \ \operatorname{tr} f = \int_\mathds{T} {\varphi(t) \psi(t) dt} = (\varphi,\psi)_{L_2(\mathds{T})},
\]
and this means that the equality~\eqref{eq6} is equivalent to the expression~\eqref{eq3}, which holds for a trace class operator $\operatorname{F}$ with some kernel $f$ according to Theorem~\ref{thm1}.

Next, we can conclude that the integral operator $\operatorname{F}$ is a trace class operator if
\[
  f(t,\tau) = g(t,\tau) + g^*(t,\tau), \ \ \ \varphi(t) = t^{n_1}, \ \ \ \psi(\tau) = \tau^{n_2}, \ \ \ n_1,n_2 = 0,1,2,\ldots
\]

Indeed, we obtain the simplest kernel $f(t,\tau) = (t \tau)^n$ (it corresponds to a trace class operator) if $n = n_1 = n_2$. When $n_1 \neq n_2$, the desired result follows from Proposition~\ref{prp1}, i.e., for the function~\eqref{eq4}, we assume $n = n_1 < n_2 = m + n$, and for the function~\eqref{eq5}, we assume $m + n = n_1 > n_2 = n$.

Thus, functions $\varphi,\psi$ can be polynomials, including Legendre polynomials (trace class operators form a linear space~\cite{Ref9}). Arbitrary functions from $L_2(\mathds{T})$ are approximated with a desired accuracy by polynomials. In this case, the function $f$ is represented as a linear combination of functions of the form $(t \tau)^n$, as well as functions~\eqref{eq4} and~\eqref{eq5}.

Finally, let $\varphi,\psi \in L_2(\mathds{T})$ and
\[
  \varphi = \lim\limits_{n \to \infty} \varphi_n, \ \ \ \psi = \lim\limits_{n \to \infty} \psi_m,
\]
where
\[
  \varphi_n(t) = \sum\limits_{i=0}^n \Phi_i \hat{P}_i(t), \ \ \ \psi_m(\tau) = \sum\limits_{i=0}^m \Psi_i \hat{P}_i(\tau), \ \ \ n,m \in \{0,1,2,\ldots\},
\]
and $\Phi_i,\Psi_i$ are expansion coefficients of functions $\varphi,\psi$, respectively, relative to Legendre polynomials $\{\hat P_i\}$. Then we establish the following equality for arbitrary $n$ and $m$:
\begin{equation}\label{eq7}
  \sum\limits_{i=0}^\infty \left[ \, \int_\mathds{T} {\varphi_n(t) q_i(t) \int_{t_0}^t {\psi_m(\tau) q_i(\tau) d\tau} dt} + \int_\mathds{T} {\psi_m(t) q_i(t) \int_{t_0}^t {\varphi_n(\tau) q_i(\tau) d\tau} dt} \right] = (\varphi_n,\psi_m)_{L_2(\mathds{T})},
\end{equation}
where the series on the left-hand side of the last expression converges absolutely, and its sum does not depend on the choice of a basis $\{q_i\}$.

By fixing $n$ in the formula~\eqref{eq7}, we have the bounded and therefore continuous linear functional in $L_2(\mathds{T})$, which is defined by the function $\varphi_n$ (on the left-hand side of the formula, it is a continuous functional in the space of trace-class operators~\cite{Ref9}, but it can be considered as a linear functional in the space $L_2(\mathds{T})$ using the extension by continuity~\cite{Ref10}). Putting $m \to \infty$, we obtain the continuous linear functional defined by the function $\psi$:
\[
  \sum\limits_{i=0}^\infty \left[ \, \int_\mathds{T} {\varphi_n(t) q_i(t) \int_{t_0}^t {\psi(\tau) q_i(\tau) d\tau} dt} + \int_\mathds{T} {\psi(t) q_i(t) \int_{t_0}^t {\varphi_n(\tau) q_i(\tau) d\tau} dt} \right] = (\varphi_n,\psi)_{L_2(\mathds{T})}.
\]

Next, putting $n \to \infty$, we obtain the expression~\eqref{eq3}, which proves the equality~\eqref{eq6}.
\end{proof}

The proven theorem allows one to justify that functions
\[
  f(t_1,\dots,t_k) = \psi_1(t_1) \dots \psi_k(t_k) 1(t_k-t_{k-1}) \dots 1(t_2-t_1)
\]
belong to the space $L_2^{\text{tr}(j_1 \dots j_k)}(\mathds{T}^k)$ defined in~\cite{Ref4}. This class of functions is characterized by the condition that there exist matrix traces for their expansion coefficients (they are numbered by $k$ indices) for pairs of indices for which the corresponding numbers $j_1,\dots,j_k$ coincide, and traces should not depend on the choice of a basis. These matrix traces, i.e., convolutions of expansion coefficients, define integral traces of the function $f$ over the corresponding pairs of variables (correspondence is understood in the sense of the order of numbers $j_1,\dots,j_k$, indices $i_1,\dots,i_k$ that are used for expansion coefficients, and variables $t_1,\dots,t_k$).

By applying Theorem~\ref{thm2}, we can show that matrix traces for any pair of neighbor indices $i_l,i_{l+1}$, where $l = 1,\dots,k-1$, define expansion coefficients of the function of $k - 2$ variables but with the same structure as the function $f$, which allows applying Theorem~\ref{thm2} ``iteratively.'' And matrix traces for any pair of indices $i_l,i_m$, where $l \in \{1,\dots,k - 2\}$, $m \in \{3,\dots,k\}$, $m - l > 1$, are equal to zero, and the corresponding integral traces are the zero function of $k - 2$ variables.

The importance of this result is related to the fact that the multiple Stratonovich stochastic integral for such a function is the iterated stochastic integral, which corresponds to a $k$th-order system of linear stochastic differential equations, and $\psi_1,\dots,\psi_k$ are weight functions. Solving such systems, i.e., modeling of iterated Stratonovich stochastic integrals, is necessary for constructing numerical methods to solve systems of nonlinear stochastic differential equations with high orders of strong convergence~\cite{Ref1, Ref2}.

Note that Theorem~\ref{thm2} can be proved by assuming that $L_2(\mathds{T}^2)$ is the space of complex-valued square-integrable functions. This entails a change in the formula defining the inner product. In addition, in the proof of Theorem~\ref{thm2}, one can represent the functions $\varphi,\psi \in L_2(\mathds{T})$ by orthogonal expansions with complex exponential functions and use the following result instead of Proposition~\ref{prp1}.

\begin{Proposition}\label{prp2}
The linear operator $\operatorname{F}$ with the kernel
\[
  f(t,\tau) = \me^{\mi n t} \me^{\mi m \tau} 1(t-\tau) + \me^{\mi n \tau} \me^{\mi m t} 1(\tau-t) = f(\tau,t),
\]
where $m,n \in \mathds{Z}$, $m \neq 0$, $\mi = \sqrt{-1}$, is a trace class operator.
\end{Proposition}

For brevity, we omit the proof, since it is similar to the proof of Proposition~\ref{prp1}, and it is based on the representation~\eqref{eq2}.

\section*{Appendix}

\small

The theory of orthogonal expansions of multiple and iterated Stratonovich stochastic integrals is developing. Since the publication of this article, the new results have appeared:
\begin{itemize}
  \item Kuznetsov, D.F. A new approach to the series expansion of iterated Stratonovich stochastic integrals with respect to components of the multidimensional Wiener process. The case of arbitrary complete orthonormal systems in Hilbert space. {\em Differ. Uravn. Protsesy Upr.} {\bf 2024}, {\em 2}, 73--170. \\ \url{https://doi.org/10.21638/11701/spbu35.2024.206}
  \item Kuznetsov, D.F. A new approach to the series expansion of iterated Stratonovich stochastic integrals with respect to components of the multidimensional Wiener process. The case of arbitrary complete orthonormal systems in Hilbert space. II. {\em Differ. Uravn. Protsesy Upr.} {\bf 2024}, {\em 4}, 104--190. \\ \url{https://doi.org/10.21638/11701/spbu35.2024.406}
  \item Kuznetsov, D.F. A new approach to the series expansion of iterated Stratonovich stochastic integrals with respect to components of the multidimensional Wiener process. The case of arbitrary complete orthonormal systems in Hilbert space. III. {\em Differ. Uravn. Protsesy Upr.} {\bf 2025}, {\em 3}, 118--153. \\ \url{https://doi.org/10.21638/11701/spbu35.2025.308}
\end{itemize}

The considered approach is directly applied in the following author's article:
\begin{itemize}
  \item Rybakov, K. On traces of linear operators with symmetrized Volterra-type kernels. {\em Symmetry} {\bf 2023}, {\em 15(10)}, 1821. \url{https://doi.org/10.3390/sym15101821}
\end{itemize}

\end{document}